\documentclass[reqno,11pt]{amsart}
\usepackage[margin=1in]{geometry}
\usepackage{amsthm, amsmath, amssymb, bm}
\usepackage[dvipsnames]{xcolor}
\usepackage{tikz}
\usepackage[utf8]{inputenc}
\usepackage[T1]{fontenc}
\usepackage[textsize=small,backgroundcolor=orange!20]{todonotes}
\usepackage[hidelinks]{hyperref}
\usepackage{url}
\usepackage[noabbrev,capitalize]{cleveref}
\crefname{equation}{}{}
\usepackage{xcolor,graphics}
\usepackage{asymptote}
\usepackage{float}

\newtheorem{theorem}{Theorem}[section]
\newtheorem{proposition}[theorem]{Proposition}
\newtheorem{lemma}[theorem]{Lemma}

\newtheorem{corollary}[theorem]{Corollary}

\newtheorem*{question*}{Question} \Crefname{question}{Question}{Questions}

\theoremstyle{definition}

\newtheorem{question}[theorem]{Question}

\theoremstyle{remark}

\newcommand{\I}{\mathcal{I}}
\newcommand{\D}{\mathcal{D}}

\title{The Relation Between Composability and Splittability of Permutation Classes}

\author[Zhang]{Rachel Yun Zhang}
\address{Department of Mathematics, MIT, Cambridge, MA 02139, USA}
\email{rachelyz44@gmail.com}

\date{July 2019}

\begin{document}

\maketitle

\begin{abstract}
    A permutation class $C$ is said to be \emph{splittable} if there exist two proper subclasses $A, B \subsetneq C$ such that any $\sigma \in C$ can be red-blue colored so that the red (respectively, blue) subsequence of $\sigma$ is order isomorphic to an element of $A$ (respectively, $B$). The class $C$ is said to be \emph{composable} if there exists some number of proper subclasses $A_1, \dots, A_k \subsetneq C$ such that any $\sigma \in C$ can be written as $\alpha_1 \circ \dots \circ \alpha_k$ for some $\alpha_i \in A_i$. We answer a question of Karpilovskij by showing that there exists a composable permutation class that is not splittable. We also give a condition under which an infinite composable class must be splittable.

\end{abstract}

\section{Introduction}

The study of permutation classes, i.e. hereditary sets of permutations, is motivated naturally by the study of pattern avoidance. Many techniques have been applied to count the permutations of length $n$ that avoid a certain set of permutations. While the set of permutations that avoid a single permutation of length $3$ and those that avoid certain length $4$ permutations have structural characterizations, such methods were not applicable to counting the number of permutations that avoid $1324$. In 2012, Claesson, Jel\'inek, and Steinr\'imsson showed that any permutation $\pi$ avoiding $1324$ can be written as the merge of a permutation avoiding $132$ and another avoiding $213$. Using this result, they bounded above the number of $1324$-avoiding permutations of length $n$ by $16^n$. 

Consequently, the concept of \emph{splittability} was introduced in 2015 by Jel\'inek and Valtr \cite{splittability} to better understand when permutation classes can be written as the merge of two subclasses. Jel\'inek and Valtr's main focus is on principal classes, i.e. sets of all permutations avoiding a single permutation. Albert and Jel\'inek \cite{unsplittable-separable} continue the study of splittability in a more limited context, by characterizing the unsplittable subclasses of the separable permutations. 


We will be interested in the relationship between splittability and a different property of permutation classes known as \emph{composability}. The concept of the composability of permutation classes was introduced by Karpilovskij \cite{composability} in 2019 in order to relate the group structure of permutations with the hereditary structure of permutation classes. The composition of permutation classes had previously been studied in relation to sorting machines \cite{comp-sort, comp-enum}, but this was the first time permutation classes were related to the composition of subclasses.

In \cite{composability}, Karpilovskij finds instances of permutation classes that are composable and splittable, uncomposable and splittable, and uncomposable and unsplittable, but does not find one that is composable and unsplittable. He therefore asks if there is a such a class. In this paper, we answer this question in the affirmative. We also give a condition under which composability implies splittability: we show that infinite composable classes that avoid either an increasing or decreasing permutation must be splittable. 

An outline of this paper is as follows. In Section 2, we define the concepts of splittability and composability in more depth. In Section 3, we present our results on the relationship between composability and splittability. Finally, in Section 4, we state some open questions.

\section{Preliminaries}

\subsection{Permutation Classes}

A \emph{permutation} is a sequence $\pi$ of distinct numbers $\pi_1, \dots, \pi_n \in [n]$. In this case, $n$ is the \emph{length} of $\pi$, and we denote this by $|\pi| = n$. When writing permutations, we will omit commas: the permutation $1, 3, 2, 4$ is the same as $1324$. We will denote the increasing permutation of length $n$ by $\iota_n = 12\dots n$ and the decreasing permutation of length $n$ by $\delta_n = n(n-1)\dots 1)$.

Two sequences $\pi$ and $\pi'$ of length $n$ are \emph{order isomorphic} if for all $i, j \in [n]$, $\pi_i > \pi_j \Longleftrightarrow \pi'_i > \pi'_j$. We say that a permutation $\pi = \pi_1 \cdots \pi_n$ \emph{contains} a permutation $\sigma = \sigma_1 \cdots \sigma_m$ if there exist indices $1 \le i_1 < \dots < i_m \le n$ such that $\pi_{i_1} \pi_{i_2} \cdots \pi_{i_m}$ is order isomorphic to $\sigma$. If $\pi$ contains $\sigma$, we write $\sigma \le \pi$. If $\pi$ does not contain $\sigma$, we say that $\pi$ \emph{avoids} $\sigma$.

A set $C$ of permutations is \emph{hereditary} if for all $\pi \in C$ and $\sigma \le \pi$, the permutation $\sigma$ is also in $C$. We refer to a hereditary set of permutations as a \emph{permutation class}. Note that any permutation class is equal to the set of permutations avoiding a fixed set set of permutations. 

Some common permutation classes include $\I = \{ \iota_n : n \in \mathbb{N} \}$, the set of all increasing permutations, and $\D = \{ \delta_n : n \in \mathbb{N} \}$, the set of all decreasing permutations. It follows from the Erd\H{o}s-Szekeres theorem \cite{es} that any infinite permutation class must either contain $\I$ or $\D$. Some other classes we will use are $\I_m$, the class of all permutations that do not contain $\delta_{m+1}$, and $\D_m$, the class of all permutations that do not contain $\iota_{m+1}$.

\subsection{Direct Sums and Reversals}

One way to form permutations from others is the \emph{direct sum}: given permutations $\alpha$ of length $k$ and $\beta$ of length $\ell$, we define $\alpha \oplus \beta$ to be the permutation $\sigma$ of length $k + \ell$ such that $\sigma_i = \alpha_i$ for $i \in [k]$, and $\sigma_j = \beta_{j - k} + k$ if $j \in [k + \ell] \backslash [k]$. 

Given a permutation $\pi = \pi_1 \cdots \pi_n$, we let $\pi^r = \pi_n \cdots \pi_1$ denote its \emph{reversal}. Note for instance that $\iota_k^r = \delta_k$. 

\subsection{Merging, Splittability, and Atomicity}

A permutation $\sigma$ is a \emph{merge} of two permutations $\alpha$ and $\beta$ if one can color the elements of $\sigma$ red and blue such that the red subsequence is order isomorphic to $\alpha$ and the blue subsequence is order isomorphic to $\beta$. 

Following this definition, we can define the \emph{merge} of two permutation classes $A$ and $B$ to be the set of all pairwise merges:
\[
    A \odot B = \{ \sigma\ |\ \text{there exist $\alpha \in A, \beta \in B$ such that $\sigma$ is a merge of $\alpha$ and $\beta$} \}.
\]
For instance, the class $\I_m$ is equal to the merge of $m$ copies of $\I$.

We say that a permutation class $C$ is \emph{splittable} if there exist $k$ proper subclasses $A_1, \dots, A_k \subsetneq C$ such that $C \subseteq A_1 \odot A_2 \odot \dots \odot A_k$. Note that if we did not require $A_i$ to be proper subclasses, but instead allowed them also to be $C$, then $C \subseteq C \odot A_2 \odot \dots \odot A_k$ for any choices of subclasses $A_2, \dots, A_k$, which would make splittability a rather trivial condition. Hence, we require $A_i \not= C$. 

In fact, $C$ is splittable if and only if there exist \emph{two} subclasses $A, B \subsetneq C$ such that $C \subseteq A \odot B$. To see this, consider a split of $C$ into $A_1 \odot \dots \odot A_k$, and suppose that this split is \emph{irreducible}, that is, there doesn't exist a proper subset $A_{i_1}, \dots, A_{i_m}$ of the original split such that $C \subseteq A_{i_1} \odot \dots \odot A_{i_m}$. Then, $C \subseteq A_1 \odot (A_2 \odot \dots \odot A_k \cap C)$. Since $A_2 \odot \dots \odot A_k$ does not contain $C$ by the irreducibility assumption, the set $A_2 \odot \dots \odot A_k \cap C$ is a proper subclass of $C$, so $C$ is splittable into two subclasses.

Finally, we say that a class $C$ is \emph{atomic} if for any $\alpha, \beta \in C$, there exists some $\sigma \in C$ such that $\alpha \le \sigma$ and $\beta \le \sigma$. 

\subsection{Inflation and Unsplittability}

We first give a separate, equivalent condition for splittability.

\begin{proposition}[\cite{splittability}, Lemma 1.1] \label{prop:split-equiv}
    A permutation class $C$ is splittable if and only if there exist two elements $\pi, \pi' \in C$ such that for any $\sigma \in C$, there exists a red-blue coloring of $\sigma$ such that the red part avoids $\pi$ and the blue part avoids $\pi'$. Alternatively, $C$ is unsplittable if and only if for all $\pi, \pi' \in C$, there exists $\sigma \in C$ such that any red-blue coloring of $\sigma$ either has the red part contain $\pi$ or the blue part contain $\pi'$.
\end{proposition}

\begin{corollary}[\cite{splittability}] \label{cor:atomicity}
    If $C$ is unsplittable, then it is atomic.
\end{corollary}

Given a permutation $\pi$ of length $n$ and $n$ permutations $\sigma_1, \dots, \sigma_n$, the \emph{inflation} of $\pi$ by $\sigma_1, \dots, \sigma_n$, denoted $\pi[\sigma_1, \dots, \sigma_n]$, is equal to the unique permutation $\overline{\sigma}_1\overline{\sigma}_2\dots \overline{\sigma}_n$ such that $\sigma_i$ is order isomorphic to $\overline{\sigma}_i$ for all $i \in [n]$, and for any $i, j \in [n]$, all elements of $\overline{\sigma}_i$ are greater than all elements of $\overline{\sigma}_j$ if and only if $\pi_i > \pi_j$. We denote the inflation of $\pi$ by $n$ copies of a permutation $\sigma$ by $\pi[\sigma]$.

\begin{figure}[h!]
\centering
\includegraphics[scale=0.55]{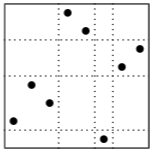}
\caption{The inflation of the permutation $2413$ by $132$, $21$, $1$, and $12$ is $24387156$.}
\label{fig:inflation}
\end{figure}

We can extend the definition of inflation to permutation classes. That is, for permutation classes $A$ and $B$,
\[
    A[B] = \{ \pi[\sigma] | \pi \in A, \sigma \in B \}.
\]

The following lemma allows us to inflate unsplittable classes to attain another unsplittable class. This lemma was first proved in \cite{unsplittable-separable}, but we will provide the proof here for completeness.

\begin{lemma} \label{lemma:inflate_unsplit}
    If permutation classes $A$ and $B$ are unsplittable, then $A[B]$ is also unsplittable.
\end{lemma}

\begin{proof}
    We will use the equivalent definition of splittability given in Proposition~\ref{prop:split-equiv}. Suppose $A[B]$ is splittable. Then, there exist $\pi, \pi' \in A[B]$ such that any $\sigma \in A[B]$ has a red-blue coloring with the red part avoiding $\pi$ and blue part avoiding $\pi'$. 
    
    Since $\pi, \pi' \in A[B]$, we can write $\pi = \sigma[\tau_1, \dots, \tau_n]$ and $\pi' = \sigma'[\tau'_1, \dots, \tau'_m]$ with $\sigma, \sigma' \in A$ and $\tau_i, \tau'_i \in B$. Since $B$ is unsplittable and therefore atomic by Corollary~\ref{cor:atomicity}, there exists some $\tau \in B$ that contains all $\tau_1, \dots, \tau_n$ and some $\tau' \in B$ that contains all $\tau'_1, \dots, \tau'_m$. Since $\sigma[\tau_1, \dots, \tau_n] \le \sigma[\tau]$, we may assume $\pi = \sigma[\tau]$, and similarly we may assume $\pi' = \sigma'[\tau']$.
    
    Since $A$ and $B$ are unsplittable, there exists $\sigma^+ \in A$ such that any red-blue coloring of $\sigma^+$ contains either a red $\sigma$ or a blue $\sigma'$. Similarly there exists $\tau^+ \in B$ such that any red-blue coloring of $\tau^+$ contains either a red $\tau$ or blue $\tau'$. 
    
    We claim that any red-blue coloring of $\sigma^+[\tau^+] \in A[B]$ contains either a red $\pi$ or a blue $\pi'$. To see this, consider a red-blue coloring of $\sigma^+[\tau^+]$. Using this coloring, we red-blue color $\sigma^+$ and $\tau^+$ as follows: color $\sigma^+_i$ red if the corresponding copy of $\tau^+$ contains a red $\tau$ and blue otherwise (in which case it must contain a blue $\tau'$). Then if there is a red $\sigma$ in the constructed coloring of $\sigma^+$, there is a red copy of $\pi$ in $\sigma^+[\tau^+]$. Otherwise there is a blue copy of $\pi'$ in the constructed coloring of $\sigma^+$, which implies that there is a blue copy of $\pi'$ in $\sigma^+[\tau^+]$, a contradiction. 
\end{proof}

\subsection{Composability}

The \emph{composition} of two permutations $\alpha$ and $\beta$ of length $n$ is the permutation $\sigma$ of length $n$, where $\sigma_i = \alpha_{\beta_i}$. We denote the composition of $\alpha$ and $\beta$ by $\alpha \circ \beta$.

The \emph{composition} of two classes $A$ and $B$ is  
\[
    A \circ B = \{ \alpha \circ \beta\ |\ \alpha \in A, \beta \in B, |\alpha| = |\beta| \}.
\]
We say that a permutation class $C$ is \emph{$k$-composable} if there exist $k$ proper subclasses $A_1, \dots, A_k \subsetneq C$ such that $C \subseteq A_1 \circ \dots \circ A_k$. If $C$ is $k$-composable for some $k$, we say that $C$ is \emph{composable}.

Note that unlike splittability, $k$-composability does not necessarily imply $2$-composability. Furthermore, unlike merging, the composition of two classes is not commutative.

\section{The Relation Between Composability and Splittability}

In \cite{composability}, Karpilovskij investigates whether many specific classes are composable. Among other classes, he shows that $\mathcal{L} = \I[\D]$, the class of layered permutations, is uncomposable. It follows from Lemma~\ref{lemma:inflate_unsplit} that $\mathcal{L}$ is unsplittable. Furthermore, letting $\mathcal{L}_k = \iota_k[\D]$ be the class of layered permutations with at most $k$ layers, he shows that $\mathcal{L}_2$ and $\mathcal{L}_3$ are examples of uncomposable yet splittable classes, and $\mathcal{L}_k$ with $k \ge 4$ are examples of composable and splittable classes. However, in the classes Karpilovskij studied, he did not find a composable and unsplittable class. Motivated by this, he asks if there is a permutation class that is composable and unsplittable. Here we demonstrate such a permutation class, answering his question.

\begin{theorem}
    The permutation class $C = \I[\D[\I]]$ is composable but also unsplittable.
\end{theorem}

\begin{proof}
    Note that $\I$ and $\D$ both consist of exactly one permutation of each length. Therefore there are no proper infinite subclasses, so $\I$ and $\D$ are unsplittable. It then follows from Lemma~\ref{lemma:inflate_unsplit} that $C$ is also unsplittable. 
    
    Note that the class of layered permutations, $\mathcal{L} = \I[\D]$, is a strict subclass of $C$. We claim that $C \subseteq \mathcal{L} \circ \mathcal{L}$. To see this, we can write any element $\pi \in C$ as $\pi_1 \oplus \dots \oplus \pi_n$, where each $\pi_i \in \D[\I]$. Then, define $\alpha = \pi_1^r \oplus \dots \oplus \pi_n^r$ and $\beta = \delta_{k_1} \oplus \dots \oplus \delta_{k_n}$, where $k_i = |\pi_i|$. Both $\alpha$ and $\beta$ are layered permutations, and $\pi = \alpha \circ \beta$ (see Figure~\ref{fig:idi-comp}). It follows that $C \subset \mathcal{L} \circ \mathcal{L}$, so $C$ is composable.
\end{proof}

\begin{figure}[h!]
\centering
\includegraphics[scale=0.55]{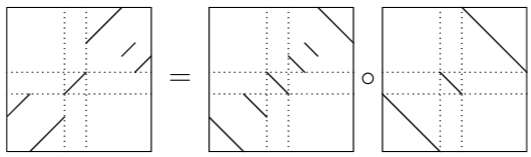}
\caption{The composability of $I[D[I]]$.}
\label{fig:idi-comp}
\end{figure}

Although composability does not imply splittability in general, we may give some conditions under which it does. In order to do this, we will need to prove a lemma about the maximum lengths of decreasing sequences in the composition of permutation classes, which was first proved in \cite{composability}.

\begin{lemma} \label{lemma:decreasing}
    If $A \subseteq \I_k$ and $B \subseteq \I_\ell$ for some $k, \ell \ge 1$, then $A \circ B \subseteq \I_{k\ell}$. Furthermore, if $A \subseteq \D_k$ and $B \subseteq \D_\ell$, then $A \circ B \subseteq \I_{k\ell}$.
\end{lemma}

\begin{proof}
    First, we handle the case that $A \subseteq \I_k$ and $B \subseteq \I_\ell$. Take $\alpha \in A$ and $\beta \in B$, and let $\gamma = \alpha \circ \beta$. Consider a maximal decreasing sequence in $\gamma$, given by the indices $i_1, \dots, i_m$, and consider $\alpha'$, the restriction of $\alpha$ to the indices $\beta_{i_1}, \dots, \beta_{i_m}$, and $\beta'$, the restriction of $\beta$ to the indices $i_1, \dots, i_m$. Then, $\delta_m = \alpha' \circ \beta'$. Note that for any indices $u, v \in [m]$ such that $u < v$, $\beta'_u < \beta'_v$ if and only if $\alpha'_{\beta'_u} > \alpha'_{\beta'_v}$. That is, for any increasing sequence in $\beta'$, there must be a corresponding decreasing sequence in $\alpha'$. Then, the length of the maximal increasing sequence in $\beta'$ is at most $k$, and the length of the maximal decreasing sequence in $\beta'$ is at most $\ell$, giving that $m = |\beta'| \le k\ell$ by the Erd\H{o}s-Szekeres theorem \cite{es}.
    
    The second case, that $A \subseteq \D_k$ and $B \subseteq \D_\ell$, follows from an analogous argument.
\end{proof}

\begin{theorem}
    If an infinite composable permutation class $C$ avoids either an increasing permutation or a decreasing permutation, then it is splittable.
\end{theorem}

\begin{proof}
    Consider first the case that $C$ avoids a decreasing permutation of length $m+1$. Since it is composable, there exist $k$ subclasses $A_1, \dots, A_k \subsetneq C$ such that $C \subseteq A_1 \circ \dots \circ A_k$. Noting that $A_i \subseteq C \subseteq \I_m$, this gives that $C \subseteq A_1 \circ \I_m \circ \dots \circ \I_m$, where there are $k-1$ copies of $\I_m$. By Lemma~\ref{lemma:decreasing}, $\I_m \circ \dots \circ \I_m \subseteq \I_{m^{k-1}}$, so $C \subseteq A_1 \circ \I_{m^{k-1}}$. Composing a permutation $\pi$ with an element of $\I_n$ is equivalent to de-merging $\pi$ into up to $n$ subpermutations, then re-merging them (see Figure~\ref{fig:merge}). Using the fact that for any $\pi \in A_1$, $A_1$ contains all of $\pi$'s subpermutations, we see that $C \subseteq \underbrace{A_1 \odot \dots \odot A_1}_{m^{k-1}}$.

    \begin{figure}[h!]
    \centering
    \includegraphics[scale=0.4]{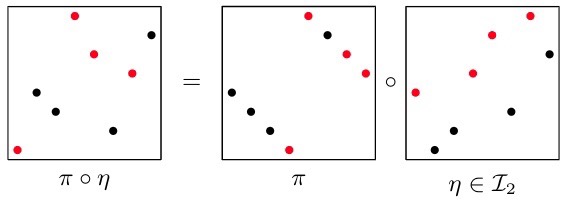}
    \caption{Composition with an element of $\I_2$. Composition of $\pi$ with $\eta \in \I_2$ de-merges $\pi$ into the red and blue subsequences and remerges them as shown in $\pi \circ \eta$.}
    \label{fig:idi-comp}
    \end{figure}
    
    The second case is that $C$ avoids an increasing permutation of length $m+1$. Let $C \subseteq A_1 \circ \dots \circ A_k$, where $A_i$ are proper subclasses of $C$. Note that if $k$ is even, then $A_1 \circ \dots \circ A_k$ does not contain $\delta_{m^k+1}$, which implies that $C$ is finite, a contradiction. Therefore, $k$ is odd. Next, pair up $A_{2i}$ with $A_{2i+1}$ for $i \in [\frac{k-1}{2}]$, and we see that $A_{2i} \circ A_{2i+1} \subseteq \I_{m^2}$ by Lemma~\ref{lemma:decreasing}. Then, $C \subseteq A_1 \circ \dots \circ A_k \subseteq A_1 \circ \I_{m^{k-1}}$, and once again we have that $C$ is the merge of $m^{k-1}$ copies of $A_1$. 
\end{proof}

\section{Further Questions}

As noted in \cite{composability}, and also as a corollary of Lemma~\ref{lemma:decreasing} and the fact that any infinite class $C$ must contain at least one of $\I$ and $\D$, any infinite class $C$ not containing $\D$ cannot be $2k$-composable for any $k \in \mathbb{N}$. However, we do not know of any general criteria that ensure that a class is not $3$-composable. In \cite{composability}, the techniques used to show uncomposability include counting arguments and quantifications of the similarity between certain composition of classes and better behaved classes, but besides the above stated condition for $2k$-uncomposability, we do not have methods that prove uncomposability for general classes. 

\begin{question}
    Give a condition under which a class $C$ is not $k$-composable for odd $k$.
\end{question}

Even more generally, we do not have general criteria under which a class is not composable.

\begin{question}
    Give a condition under which a class $C$ is not composable.
\end{question}

Recall that a class $C$ is splittable if there exist two proper subclasses $A, B \subsetneq C$ such that $C$ is a subset of $A \odot B$. We suggest a modified concept: \emph{exact-splittability}, which means that there exist two sets $A, B \subsetneq C$ such that $C = A \odot B$. Once we have an exact-split of $C$ into $A$ and $B$, we can repeat the process on $A$ and $B$, to decompose into even more subclasses. This process is guaranteed to terminate as long as $C$ is not the class of all permutations: given that $C$ avoids a permutation of length $m + 1$, it is not hard to show that if $C = A \odot B$ with $A, B \subsetneq C$, then $A$ and $B$ must both avoid a permutation of length $m$. Therefore, via repeated exact-splittings, we can eventually write $C = A_1 \odot \dots \odot A_k$, where each $A_i$ is not exact-splittable. We call such a set $\{ A_1, \dots, A_k \}$ of proper subclasses of $C$ an \emph{irreducible exact-splitting of $C$}.

\begin{question}
    Is the irreducible exact-splitting of an infinite permutation class $C$ unique?
\end{question}

Note that if we allow finite classes, the class $C$ of all permutations of length up to $6$ excluding $\iota_6$ and $\delta_6$ is not uniquely decomposed into an irreducible exact-splitting. If we let $g(\Pi)$ be the set of permutations $\sigma$ such that $\sigma \le \pi$ for some $\pi \in \Pi$, where $\Pi$ is a set of permutations, then $C = g(1) \odot g(1) \odot g(12) \odot g(21) = g(1) \odot g(1) \odot g(1) \odot g(132, 213, 231, 312)$, which are two different irreducible exact-splittings.

\section{Acknowledgements}

This research was conducted at the University of Minnesota Duluth REU and was supported by NSF / DMS grant 1659047 and NSA grant H98230-18-1-0010.
I would like to thank Joe Gallian, who suggested the problem, and Colin Defant and Caleb Ji for their suggestions on this paper.

\end{document}